\documentclass[graybox,envcountsect,envcountsame]{svmult}

\usepackage{mathptmx}       
\usepackage{helvet}         
\usepackage{courier}        
\usepackage{type1cm}        
\usepackage{makeidx}         
\usepackage{graphicx}        
\usepackage{multicol}        
\usepackage[bottom]{footmisc}

\usepackage[mathscr]{euscript}

\makeindex             

\usepackage{amssymb,amsmath,amsxtra}
\usepackage[all]{xy}
\usepackage{comment}
\usepackage[lofdepth,lotdepth]{subfig}
\usepackage{color}

\DeclareMathOperator{\Aut}{Aut}

\DeclareMathOperator{\Gal}{Gal}
\DeclareMathOperator{\impart}{Im}

\DeclareMathOperator{\M}{M}

\DeclareMathOperator{\nrd}{nrd}

\DeclareMathOperator{\PSL}{PSL}

\DeclareMathOperator{\SL}{SL}

\newcommand{\C}{\mathbb C}

\newcommand{\Q}{\mathbb Q}
\newcommand{\R}{\mathbb R}
\newcommand{\Z}{\mathbb Z}
\newcommand{\Qbar}{\overline{\mathbb Q}}

\newcommand{\frakl}{\mathfrak{l}}
\newcommand{\frakN}{\mathfrak{N}}
\newcommand{\frakp}{\mathfrak{p}}

\newcommand{\frakD}{\mathfrak{D}}

\newcommand{\calH}{\mathscr{H}}
\newcommand{\calD}{\mathscr{D}}
\newcommand{\calO}{\mathscr{O}}

\newcommand{\psmod}[1]{~(\textup{\text{mod}}~{#1})}

\newcommand{\quat}[2]{\displaystyle{\biggl(\frac{#1}{#2}\biggr)}}

\newcommand{\eps}{\epsilon}

\newcommand{\la}{\langle}
\newcommand{\ra}{\rangle}

\begin{document}

\title*{Computing power series expansions of modular forms}
\author{John Voight and John Willis}
\institute{John Voight \at Department of Mathematics and Statistics, University of Vermont, 16 Colchester Ave, Burlington, VT 05401, USA, \email{jvoight@gmail.com}
\and John Willis \at Department of Mathematics and Statistics, University of South Carolina, 1523  Greene St, Columbia, SC 29205, USA \email{jmart.will@gmail.com}}
%
%
\maketitle

\abstract*{We exhibit a method to numerically compute power series expansions of modular forms on a cocompact Fuchsian group, using the explicit computation of a fundamental domain and linear algebra.  As applications, we compute Shimura curve parametrizations of elliptic curves over a totally real field, including the image of CM points, and equations for Shimura curves.}	

\abstract{We exhibit a method to numerically compute power series expansions of modular forms on a cocompact Fuchsian group, using the explicit computation a fundamental domain and linear algebra.  As applications, we compute Shimura curve parametrizations of elliptic curves over a totally real field, including the image of CM points, and equations for Shimura curves.}	

\section{Introduction}

A classical modular form $f:\calH \to \C$ on the upper half-plane $\calH$ satisfies the translation invariance $f(z+1)=f(z)$ for $z \in \calH$, so $f$ admits a Fourier expansion (or \emph{$q$-expansion}) 
\[ f(z)=\sum_{n=0}^{\infty} a_n q^n \]
at the cusp $\infty$, where $q=e^{2\pi iz}$.  If further $f$ is a normalized eigenform for the Hecke operators $T_n$, then the coefficients $a_n$ are the eigenvalues of $T_n$ for $n$ relatively prime to the level of $f$.  The \emph{$q$-expansion principle} expresses in a rigorous way the fact that a modular form is characterized by its $q$-expansion, and for this reason (and others) $q$-expansions remain an invaluable tool in the study of classical modular forms.

By contrast, modular forms on cocompact Fuchsian groups do not admit $q$-expansions due to the lack of cusps.  A \emph{modular form} $f:\calH \to \C$ of weight $k \in 2\Z_{\geq 0}$ for a cocompact Fuchsian group $\Gamma \leq \PSL_2(\R)$ is a holomorphic map satisfying
\[ f(g z)=j(g,z)^k f(z) \]
for all $g \in \Gamma$, where $j(g,z)=cz+d$ if $g=\begin{pmatrix}a & b \\ c & d \end{pmatrix}$.  

However, not all is lost: such a modular form $f$ still admits a power series expansion in the neighborhood of a point $p \in \calH$.  Indeed, a $q$-expansion is really just a power series expansion at $\infty$ in the parameter $q$, convergent for $|q|<1$; so it is natural to consider a neighborhood of $p$ normalized so the expansion also converges in the unit disc but for a parameter $w$.  For this purpose, we map the upper half-plane conformally to the unit disc $\calD$ via the map 
\begin{align*}
w:\calH &\to \calD \\
z &\mapsto w(z)=\frac{z-p}{z-\overline{p}} 
\end{align*}
sending $p \mapsto w(p)=0$, where $\overline{\phantom{x}}$ denotes complex conjugation.  We consider the series expansion of a form $f$ of weight $k$ given by
\begin{equation*} \tag{$*$}
 f(z)=(1-w)^k \sum_{n=0}^{\infty} b_n w^n
\end{equation*}
where $w=w(z)$, convergent in the disc $\calD$ with $|w|<1$.  

There are several reasons to consider series of the form $(*)$.  First, the term 
\[ 1-w(z)=\frac{p-\overline{p}}{z-\overline{p}} \]
is natural to include as it arises from the automorphy factor of the linear fractional transformation $w(z)$.  Second, the ordinary Taylor coefficients arise from evaluating derivatives of $f$, but the derivative of a modular form of weight $k \neq 0$ is no longer a modular form!  This can be ameliorated by considering instead the differential operator introduced by Maass and Shimura: 
\[ \partial_k = \frac{1}{2\pi i}\left(\frac{d}{dz} + \frac{k}{z-\overline{z}}\right). \] 
If $f$ is a modular form of weight $k$ then $\partial_k f$ transforms like a modular form of weight $k+2$, but at the price that $f$ is now only real analytic.  The coefficients $b_n$ in the expansion $(*)$ then arise from evaluating Shimura-Maass derivatives $\partial^n f$ of $f$ at the point $p$, where we let $\partial^n = \partial_{k+2(n-1)} \circ ... \circ\partial_{k+2}\circ\partial_k$.  Finally, when $p$ is a CM point and $\Gamma$ is a congruence group, the coefficients $b_n$ are algebraic (up to a normalization factor), as shown by Shimura \cite{Shimura1,Shimura2}.  Rodriguez-Villegas and Zagier \cite{RVZ} and O'Sullivan and Risager \cite{OSR} (see also the exposition by Zagier \cite{Zagier} and further work by Bertolini, Darmon, and Prasanna \cite[\S\S 5-6]{Prasanna}) further linked the coefficients $b_n$ to square roots of central values of Rankin-Selberg $L$-series and show that they satisfy a recursive formula arising from the differential structure of the ring of modular forms.  Mori \cite{Mori1,Mori2} and Datskovsky and Guerzhoy \cite{DatsGuerz} have also studied the $p$-adic interpolation properties of the coefficients $b_n$ in the context of $p$-adic $L$-functions, with some results also in the cocompact case.  For these reasons, we will consider series expansions of the form $(*)$ in this article, which we will call \emph{power series expansions}.  

In the special case when the Fuchsian group $\Gamma$ is commensurable with a triangle group, the differential approach to power series expansions of modular forms is particularly pleasing.  For $a,b,c \in \Z_{\geq 2} \cup \{\infty\}$ with $1/a+1/b+1/c<1$, we define the $(a,b,c)$-triangle group to be the subgroup of orientation-preserving isometries in the group generated by reflections in the sides of a hyperbolic triangle with angles $\pi/a,\pi/b,\pi/c$.  In this case, power series expansions for a uniformizing function at the vertices of the fundamental triangle are obtained as the inverse of the ratio of ${}_2F_1$-hypergeometric functions.  This case was also of great classical interest, and has been taken up again more recently by Bayer \cite{Bayer}, Bayer and Travesa \cite{BayerTravesa1,BayerTravesa2}, the first author \cite{Voight-thesis}, and Baba and Granath \cite{BabaGranath}; this includes the well-studied case where the Fuchsian group arises from the quaternion algebra of discriminant $6$ over $\Q$, corresponding to the $(2,4,6)$-triangle group.

In this article, we exhibit a general method for numerically computing power series expansions of modular forms for cocompact Fuchsian groups.  It is inspired by the method of Stark \cite{Stark} and Hejhal \cite{Hejhal1,Hejhal2,Hejhal3}, who used the same basic principle to compute Fourier expansions for Maass forms on $\SL_2(\Z)$ and the Hecke triangle groups.  (There has been substantial further work in this area; see for example Then \cite{Then}, Booker, Str\"ombergsson, and Venkatesh \cite{BSV}, and the references therein.)

The basic idea is quite simple.  Let $\Gamma$ be a Fuchsian group with compact fundamental domain $D \subset \calD$ contained in a circle of radius $\rho>0$.  To find a power series expansion $(*)$ for a form $f$ on $\Gamma$ of weight $k$ valid in $D$ to some precision $\eps>0$, we consider the approximation
\[ f(z) \approx f_N(z)=(1-w)^k\sum_{n=0}^{N} b_n w^n \]
valid for $|w|\leq\rho$ and some $N \in \Z_{\geq 0}$.  Then for a point $w=w(z)$ on the circle of radius $\rho$ with $w \not\in D$, we find $g \in \Gamma$ such that $w'=g w \in D$; letting $z'=z(w')$, by the modularity of $f$ we have
\[ (1-w')^k\sum_{n=0}^N b_n (w')^n = f_N(z') \approx f(z') = j(g,z)^k f(z) \approx j(g,z)^k (1-w)^k\sum_{n=0}^N b_n w^n \]
valid to precision $\eps>0$.  For each such point $w$, this equality imposes an approximate (nontrivial) linear relation on the unknown coefficients $b_n$.  By taking appropriate linear combinations of these relations, or (what seems better in practice) using the Cauchy integral formula, we recover the approximate coefficients $b_n$ using standard techniques in linear algebra.

An important issue that we try to address in this paper is the numerical stability of this method---unfortunately, we cannot benefit from the exponential decay in the terms of a Fourier expansion as in previous work, so our numerical methods must be correspondingly more robust.  Although we cannot prove that our results are correct, there are several tests that allow one to be quite convinced that they are correct, and we show in several examples that they agree with cases that are known.  (See also Remark \ref{canweproveit} and work of Booker, Str\"ombergsson and Venkatesh \cite{BSV}, who rigorously verify the numerical computations in Hejhal's method for Maass forms.)

Nelson \cite{Nelson,Nelson2} finds power series expansions by directly computing the Shimizu lift (a realization of the Jacquet-Langlands correspondence) of a modular form on a Shimura curve over $\Q$ to a classical modular curve.  It will be interesting to compare the two techniques.  Owing to its simplicity, we believe that our method is worthy of investigation.  It has generalizations to a wide variety of settings: noncongruence groups, nonarithmetic groups (e.g.\ nonarithmetic triangle groups), real analytic modular forms, and higher dimensional groups; and it applies equally well for arithmetic Fuchsian groups with an arbitrary totally real trace field.  

Finally, this analysis at the complex place suggest that $p$-adic analytic methods for computing power series expansions would also be interesting to investigate.  For example, Franc \cite{Franc} investigates the values of Shimura-Maass derivatives of modular forms at CM points from a rigid analytic perspective.

This paper is organized as follows.  In Section 2, we introduce some basic notation and background.  Then in Section 3, we exhibit our algorithm in detail.  Finally, in Section 4 we give several examples: in the first two examples we verify the correctness of our algorithm; in the third example, we compute the Shimura curve parametrization of an elliptic curve over a totally real field and the image of CM point; in the fourth example, we show how our methods can be used to compute the equation of a Shimura curve.

\section{Preliminaries}

We begin by considering the basic setup of our algorithm; as basic references we refer to Beardon \cite{Beardon} and Katok \cite{Katok}.  

\subsection*{Fuchsian groups and fundamental domains}

Let $\Gamma$ be a Fuchsian group, a discrete subgroup of $\PSL_2(\R)$, the orientation-preserving isometries of the upper half-plane $\calH$.  Suppose that $\Gamma$ is \emph{cofinite}, so $X=\Gamma \backslash \calH$ has finite hyperbolic area; then $\Gamma$ is finitely generated.  If $X$ is not compact, then $X$ has a cusp, and the existence of $q$-expansions at cusps in many cases obviates the need to compute power series expansions separately (and for Maass forms and generalizations, we refer to the method of Hejhal \cite{Hejhal2}); more seriously, our method apparently does not work as well in the non-cocompact case (see Example 2 in Section 4).  So we suppose that $\Gamma$ is cocompact.  

We will frequently move between the upper half-plane model $\calH$ and the Poincar\'e unit disc model $\calD$ for hyperbolic space, which are conformally identified via the maps
\begin{equation} \label{phimap}
\setlength{\arraycolsep}{0.5ex}
\begin{array}{rlcrl}
w:\calH &\to \calD & \quad & z:\calD &\to \calH \\
z &\mapsto w(p;z)=\displaystyle{\frac{z-p}{z-\overline{p}}} & & w &\mapsto z(p;w)=\displaystyle{\frac{\overline{p}w-p}{w-1}}
\end{array}
\end{equation}
for a choice of point $p \in \calH$.  Via this identification, the group $\Gamma$ acts also on $\calD$, and to ease notation we identify these actions.

A Fuchsian group $\Gamma$ is \emph{exact} if there is a finite set $G \subset \SL_2(K)$ with $K \hookrightarrow \Qbar \cap \R$ a number field whose image in $\PSL_2(K) \subset \PSL_2(\R)$ generates $\Gamma$.  When speaking in an algorithmic context, we will suppose that the group $\Gamma$ is exact.  (Even up to conjugation in $\PSL_2(\R)$, not every finitely generated Fuchsian group is exact.)  Algorithms for efficiently computing with algebraic number fields are well-known (see e.g.\ Cohen \cite{Cohen}), and we will use these without further mention.

In this setting, we have the following result, due to the first author \cite{Voight-funddom}.  For a point $p \in \calH$, we denote by $\Gamma_p=\{g \in \Gamma : g(p)=p\}$ the stabilizer of $p$ in $\Gamma$.

\begin{theorem} \label{funddom}
There exists an algorithm that, given as input an exact, cocompact Fuchsian group $\Gamma$ and a point $p \in \calH$ with $\Gamma_p=\{1\}$, computes as output a fundamental domain $D(p) \subset \calH$ for $\Gamma$ and an algorithm that, given $z \in \calH$ returns a point $z' \in D(p)$ and $g \in \Gamma$ such that $z'=g z$.
\end{theorem}

In particular, in the course of exhibiting this algorithm a suite of methods for computing in the unit disc are developed.  The algorithm in Theorem \ref{funddom} has been implemented in Magma \cite{Magma}.

The fundamental domain $D(p)$ in Theorem \ref{funddom} is the \emph{Dirichlet domain} centered at $p$,
\[ D(p)=\{z \in \calH : d(z,p) \leq d(g z,p) \text{ for all $g \in \Gamma$}\} \]
where $d$ is the hyperbolic distance.  The set $D(p)$ is a closed, connected, and hyperbolically convex domain whose boundary consists of finitely many geodesic segments.  The image of $D(p)$ in $\calD$ is analogously described, and we use the same notation for it.  A domain with this description is indeed desirable for consideration of power series centered at $p$, as we collect in each orbit of $\Gamma$ the points closest to $p$.

\subsection*{Modular forms and power series expansions}

A \emph{(holomorphic) modular form} of weight $k \in 2\Z_{\geq 0}$ for $\Gamma$ is a holomorphic map $f:\calH \to \C$ with the property that
\begin{equation} \label{transform} 
f(g z)=j(g,z)^k f(z)
\end{equation}
for all $g = \pm \begin{pmatrix} a & b \\ c & d \end{pmatrix} \in \Gamma$ where $j(g,z)=cz+d$.  (Note that although the matrix is only defined up to sign, the expression $j(g,z)^k$ is well-defined since $k$ is even.  Our methods would extend in a natural way to forms of odd weight with character on subgroups of $\SL_2(\R)$, but for simplicity we do not consider them here.)  Let $M_k(\Gamma)$ be the finite-dimensional $\C$-vector space of modular forms of weight $k$ for $\Gamma$, and let $M(\Gamma)=\bigoplus_{k} M_k(\Gamma)$ be the ring of modular forms for $\Gamma$ under multiplication.

A function $f: \calH \to \C$ is said to be \emph{nearly holomorphic} if it is of the form
\[ f(z)=\sum_{d=0}^{m} \frac{f_d(z)}{(z-\overline{z})^d} \]
where each $f_d:\calH \to \C$ is holomorphic.  A nearly holomorphic function is real analytic.  A \emph{nearly holomorphic modular form} of weight $k \in 2\Z_{\geq 0}$ for $\Gamma$ is a nearly holomorphic function $f:\calH \to \C$ that transforms under $\Gamma$ as in  (\ref{transform}).  Let $M_k^*(\Gamma)$ be the $\C$-vector space of nearly holomorphic modular forms of weight $k$ for $\Gamma$ and let $M^*(\Gamma)=\bigoplus_k M_k^*(\Gamma)$; then $M_k(\Gamma) \subseteq M_k^*(\Gamma)$.

Via the identification of the upper half-plane with the unit disc, we can consider a modular form $f$ also as a function on the unit disc.  The transformation property (\ref{transform}) of $g \in \Gamma$ is then described by
\begin{equation} \label{fgw}
 f(gz) = (cz+d)^k f(z) = \left( \frac{(c\overline{p}+d)w - (p+d)}{w-1} \right)^k f(z)\qquad\text{where }z = z(w).
\end{equation}

Let $f$ be a modular form of weight $k$ for $\Gamma$.  Since $f$ is holomorphic in $\calD$, it has a \emph{power series expansion}
\begin{equation} \label{powseries}
f(z)=(1-w)^k \sum_{n=0}^{\infty} b_n w^n 
\end{equation}
with $b_n \in \C$ and $w=w(z)$, convergent in $\calD$ and uniformly convergent on any compact subset.  

The coefficients $b_n$, like the usual Taylor coefficients, are related to derivatives of $f$ as follows.  Define the \emph{Shimura-Maass differential operator} $\partial_k : M_k^*(\Gamma) \to M_{k+2}^*(\Gamma)$ by
\[ \partial_k = \frac{1}{2\pi i}\left(\frac{d}{dz} + \frac{k}{z-\overline{z}}\right). \]
Even if $f \in M_k(\Gamma)$ is holomorphic, $\partial_k f \in M_{k+2}^*(\Gamma)$ is only nearly holomorphic; indeed, 
the ring $M^*(\Gamma)$ of nearly holomorphic modular forms is the smallest ring of functions which contains the ring of holomorphic forms $M(\Gamma)$ and is closed under the Shimura-Maass differential operators.  For $n \geq 0$ let $\partial_k^n = \partial_{k+2(n-1)} \circ \cdots \circ \partial_{k+2} \circ \partial_k$, and abbreviate $\partial^n=\partial_k^n$.  We have the following proposition, proven by induction: see e.g.\ Zagier \cite[Proposition 17]{Zagier} (but note the sign error!).

\begin{lemma} \label{powcn}
Let $f:\calH\to\C$ be holomorphic at $p\in\calH$.  Then $f$ admits the series expansion \textup{(\ref{powseries})} in a neighborhood of $p$ with
$$b_n = \frac{(\partial^n f)(p)}{n!}(-4\pi y)^n$$
where $y = \impart(p)$.
\end{lemma}

\begin{remark} \label{recurrence}
This expression of the coefficients $b_n$ in terms of derivatives implies that they can also be given as (essentially) the constant terms of a sequence of polynomials satisfying a recurrence relation, arising from the differential structure on $M^*(\Gamma)$: see Rodriguez-Villegas and Zagier \cite[\S\S 6--7]{RVZ} and Zagier \cite[\S 5.4, \S 6.3]{Zagier}, who carry this out for finite index subgroups of $\SL_2(\Z)$, and more recent work of Baba and Granath \cite{BabaGranath}.  
\end{remark}

The expression for the regular derivative 
\[ D^n = \frac{1}{(2\pi i)^n} \frac{d^n}{dz^n} \] 
in terms of $\partial^n$ is given in terms of Laguerre polynomials: see Rodriguez-Villegas and Zagier \cite[\S 2]{RVZ}, Zagier \cite[(56)--(57)]{Zagier}, or O'Sullivan and Risager \cite[Proposition 3.1]{OSR}.

\begin{lemma} \label{Dvspartial}
We have
	$$\partial^nf = \sum_{r=0}^n\begin{pmatrix}n\\r\end{pmatrix}\frac{(k+r)_{n-r}}{(-4\pi y)^{n-r}}D^r f,$$
and
\[ D^n f = \sum_{r=0}^n (-1)^{n-r} \binom{n}{r} \frac{(k+r)_{n-r}}{(-4\pi y)^{n-r}} \partial^r f \]
where $y=\impart p$ and $(a)_m=a(a+1)\cdots (a+m-1)$ is the Pochhammer symbol.
\end{lemma}

\subsection*{Arithmetic Fuchsian groups}

Among the Fuchsian groups we consider, of particular interest are the arithmetic Fuchsian groups.  A basic reference is Vign\'eras \cite{Vigneras}; see also work of the first author \cite{Voight-algs} for an algorithmic perspective.

Let $F$ be a number field with ring of integers $\Z_F$.  A \emph{quaternion algebra} $B$ over $F$ is an $F$-algebra with generators $\alpha,\beta \in B$ such that 
\[ \alpha^2=a, \quad \beta^2=b, \quad \beta\alpha=-\alpha\beta \]
with $a,b \in F^\times$; such an algebra is denoted $B=\quat{a,b}{F}$. 

Let $B$ be a quaternion algebra over $F$.  Then $B$ has a unique (anti-)involution $\overline{\phantom{x}}:B \to B$ such that the \emph{reduced norm} $\nrd(\gamma)=\gamma\overline{\gamma}$ belongs to $F$ for all $\gamma \in B$.  A place $v$ of $F$ is \emph{split} or \emph{ramified} according as $B_v = B \otimes_F F_v \cong M_2(F_v)$ or not, where $F_v$ denotes the completion at $v$.  The set $S$ of ramified places of $B$ is finite and of even cardinality, and the product $\frakD$ of all finite ramified places is called the \emph{discriminant} of $B$.  

Now suppose that $F$ is a totally real field and that $B$ has a unique split real place $v \not\in S$ corresponding to $\iota_\infty:B \hookrightarrow B \otimes F_v \cong \M_2(\R)$.  An \emph{order} $\calO \subset B$ is a subring with $F\calO = B$ that is finitely generated as a $\Z_F$-submodule.  Let $\calO \subset B$ be an order and let $\calO_1^*$ denote the group of units of reduced norm $1$ in $\calO$.  Then the group $\Gamma^B(1)=\iota_\infty(\calO_1^*/\{\pm 1\}) \subset \PSL_2(\R)$ is a Fuchsian group \cite[\S\S 5.2--5.3]{Katok}.  An \emph{arithmetic Fuchsian group} $\Gamma$ is a Fuchsian group commensurable with $\Gamma^B(1)$ for some choice of $B$.  One can, for instance, recover the usual modular groups in this way, taking $F=\Q$, $\calO=M_2(\Z) \subset M_2(\Q)=B$, and $\Gamma \subset \PSL_2(\Z)$ a subgroup of finite index.  An arithmetic Fuchsian group $\Gamma$ is cofinite and even cocompact, as long as $B \not\cong M_2(\Q)$, which we will further assume.  In particular, the fundamental domain algorithm of Theorem \ref{funddom} applies.  

Let $\frakN$ be an ideal of $\Z_F$ coprime to $\frakD$.  Define
\[ \calO(\frakN)_1^\times = \{\gamma \in \calO : \gamma \equiv 1 \psmod{\frakN\calO}\} \]
and let $\Gamma^B(\frakN)=\iota_\infty(\calO(\frakN)_1^\times)$.  A Fuchsian group $\Gamma$ commensurable with $\Gamma^B(1)$ is \emph{congruence} if it contains $\Gamma^B(\frakN)$ for some $\frakN$.  

The space $M_k(\Gamma)$ has an action of Hecke operators $T_\frakp$ indexed by the prime ideals $\frakp \nmid \frakD\frakN$ that belong to the principal class in the narrow class group of $\Z_F$.  (More generally, one must consider a direct sum of such spaces indexed by the narrow class group of $F$.)  The Hecke operators can be understood as averaging over sublattices, via correspondences, or in terms of double cosets; for a detailed algorithmic discussion in this context and further references, see work of Greenberg and the first author \cite{GV} and Demb\'el\'e and the first author \cite{DV}.  The operators $T_\frakp$ are semisimple and pairwise commute so there exists a basis of simultaneous eigenforms in $M_k(\Gamma)$ for all $T_\frakp$.  

Let $K$ be a totally imaginary quadratic extension of $F$ that embeds in $B$, and let $\nu \in B$ be such that $F(\nu) \cong K$.  Let $p \in \calH$ be a fixed point of $\iota_\infty(\nu)$.  Then we say $p$ is a \emph{CM point} for $K$.  

\begin{theorem} \label{Shimuraalg}
There exists $\Omega \in \C^\times$ such that for every CM point $p$ for $K$, every congruence subgroup $\Gamma$ commensurable with $\Gamma^B(1)$, and every eigenform $f \in M_k(\Gamma)$ with $f(p) \in \Qbar^\times$, we have
\[ \frac{(\partial^n f)(p)}{\Omega^{2n}} \in \Qbar \]
for all $n \in \Z_{\geq 0}$.
\end{theorem}

The work of Shimura on this general theme of arithmetic aspects of analytic functions is vast.  For a ``raw and elementary'' presentation, see his work \cite[Main Theorem I]{Shimura1} on the algebraicity of CM points of derivatives of Hilbert modular forms with algebraic Fourier coefficients.  He then developed this theme with its connection to theta functions by a more detailed investigation of derivatives of theta functions \cite{Shimura2}.  For the statement above, see his further work on the relationship between the periods of abelian varieties with complex multiplications and the derivatives of automorphic forms \cite[Theorem 7.6]{Shimura3}.  However, these few references barely scratch the surface of this work, and we refer the interested reader to Shimura's collected works \cite{Shimuracollected} for a more complete treatment.

Note that $\Omega$ is unique only up to multiplication by an element of $\Qbar^\times$.  When $F=\Q$, the period $\Omega=\Omega_K$ can be taken as the \emph{Chowla-Selberg period} associated to the imaginary quadratic field $K$, given in terms of a product of values of the $\Gamma$-function \cite[(97)]{Zagier}. (The Chowla-Selberg theory becomes much more complicated over a general totally real field; see e.g.\ Moreno \cite{Moreno}.)  In the case where $q$-expansions are available, one typically normalizes the form $f$ to have algebraic Fourier coefficients, in which case $f(p) \in \Omega^{k}\Qbar$; for modular forms on Shimura curves, it seems natural instead to normalize the function so that its values at CM points for $K$ are algebraic.  

For the rest of this section, suppose that $\Gamma$ is a congruence arithmetic Fuchsian group containing $\Gamma^B(\frakN)$.  Then from Lemma \ref{powcn}, we see that if $f \in M_k(\Gamma)$ and as in Theorem \ref{Shimuraalg}, then the power series (\ref{powseries}) at a CM point $p$ can be rewritten as
\begin{equation} \label{normalizeOmega}
 f(z)=f(p)(1-w)^k \sum_{n=0}^{\infty} \frac{c_n}{n!} (\Theta w)^n
\end{equation}
where 
\begin{equation}
\Theta=-4\pi y \frac{(\partial f)(p)}{f(p)} = \frac{b_1}{b_0} ,
\end{equation}
and
\begin{equation}
c_n = \frac{(\partial^n f)(p)}{f(p)} \left(\frac{f(p)}{(\partial f)(p)}\right)^n = \frac{1}{n!} \frac{b_n}{b_0} \left(\frac{b_0}{b_1}\right)^n \in \Qbar 
\end{equation}
are algebraic.  In Section 3, when these hypotheses apply, we will use this extra condition to verify that our results are correct.

\begin{remark}
Returning to Remark \ref{recurrence}, for a congruence group $\Gamma$, it follows that once we have computed just a finite number of the coefficients $c_n$ for a finite generating set of $M(\Gamma)$ to large enough precision to recognize them algebraically and thus exactly, following Zagier we can compute the differential structure of $M(\Gamma)$ and thereby compute exactly all coefficients of all forms in $M(\Gamma)$ rapidly.  It would be very interesting, and quite powerful, to carry out this idea in a general context.
\end{remark}

\section{Numerical method}

In this section, we exhibit our method to compute the finite-dimensional space of (holomorphic) modular forms $M_k(\Gamma)$ of weight $k \in 2\Z_{\geq 0}$ for a cocompact Fuchsian group $\Gamma$.  Since the only forms of weight $0$ are constant, we assume $k>0$.  

\begin{remark}
To obtain meromorphic modular forms of weight $0$ (for example) we can take ratios of holomorphic modular forms of the same weight $k$.  Or, alternatively, since the Shimura-Maass derivative of a meromorphic modular form of weight $0$ is in fact meromorphic of weight $2$ ($k=0$ in the definition), it follows that the antiderivative of a meromorphic modular form of weight $2$ is meromorphic of weight $0$.  
\end{remark}

Let $\eps>0$ be the desired precision in the numerical approximations.  We write $x \approx y$ to mean approximately equal (to precision $\eps$); we leave the task of making these approximations rigorous as future work.

Let $p \in \calH$ satisfy $\Gamma_p=\{1\}$ and let $D=D(p) \subset \calD$ be the Dirichlet domain centered at $p$ as in Theorem \ref{funddom}, equipped with an algorithm that given a point $w \in \calD$ computes $g \in\Gamma$ such that $w'=g w \in D$.  Let 
\[ \rho=\rho(D)=\max\{|w| : w \in D\} \]
be the radius of $D$.  

\subsection*{Approximation by a polynomial}

The series expansion (\ref{powseries}) for a form $f \in M_k(\Gamma)$ converges in the unit disc $\calD$ and its radius of convergence is $1$.  To estimate the degree of a polynomial approximation for $f$ valid in $D$ to precision $\eps$, we need to estimate the sizes of the coefficients $b_n$.  

We observe in some experiments that the coefficients are roughly bounded.  If we take this as a heuristic, assuming $|b_n| \leq 1$ for all $n$, then the approximation
\[ f(z) \approx f_N(z)=(1-w)^k\sum_{n=0}^{N} b_n w^n \]
is valid for all $|w|\leq \rho$ with
\begin{equation} \label{valueforN}
 N = \left\lceil\frac{\log(\epsilon)}{\log \rho}\right\rceil.
\end{equation}

\begin{remark} \label{canweproveit}
As mentioned in the introduction, if $p$ is a CM point and $f$ is an eigenform for a congruence group, then the coefficient $b_n$ is related to the central critical value of the Rankin-Selberg $L$-function $L(s, f \times \theta^n)$, where $\theta$ is the modular form associated to a Hecke character for the CM extension given by $p$.  Therefore, the best possible bounds on these coefficients will involve (sub)convexity bounds for these central $L$-values (in addition to estimates for the other explicit factors that appear).
\end{remark}

In practice, we simply increase $N$ in two successive runs and compare the results in order to be convinced of their accuracy.  Indeed, a posteriori, we can go back using the computed coefficients to give a better estimate on their growth and normalize them so that the boundedness assumption is valid.  In this way, we only use this guess for $N$ in a weak way.

\subsection*{Relations from automorphy}

The basic idea now is to use the modularity of $f$, given by (\ref{transform}), at points inside and outside the fundamental domain to obtain relations on the coefficients $b_n$.  

For any point $w \in \calD$ with $|w|\leq \rho$ and $z=z(w)$, we can compute $g \in \Gamma$ such that $w'=g w \in D$; letting $z'=z(w')$, by the modularity of $f$ we have
\[  f_N(z') \approx f(z') = j(g,z)^k f(z) \approx j(g,z)^k f_N(z) \]
so
\begin{equation}
(1-w')^k\sum_{n=0}^N b_n (w')^n \approx j(g,z)^k(1-w)^k\sum_{n=0}^N b_n w^n
\end{equation}
and thus
\begin{equation}  \label{bigdealrelation}
\sum_{n=0}^N K^{\textup{a}}_n(w) b_n \approx 0
\end{equation}
where (``a'' for automorphy)
\begin{equation}
K^{\textup{a}}_n(w) = j(g,z)^k(1-w)^k w^n - (1-w')^k (w')^n.
\end{equation}
If $w$ does not belong to the fundamental domain $D$, so that $g \neq 1$, then this imposes a nontrivial relation on the coefficients $b_n$.  With enough relations, we then use linear algebra to recover the coefficients $b_n$.

\begin{remark} \label{norm}
In practice, it is better to work not with $f(z)$ but instead the \emph{normalized} function 
\[ f_{\textup{norm}}(z)=f(z) (\impart z)^{k/2} \] 
since then we have the transformation formula
\begin{equation} \label{transformabs1}
f_{\textup{norm}}(gz)=f(gz)(\impart gz)^{k/2} = j(g,z)^k f(z) \left(\frac{\impart z}{|cz+d|^2}\right)^{k/2} = \left(\frac{j(g,z)}{|j(g,z)|}\right)^k f_{\textup{norm}}(z)
\end{equation}
and the automorphy factor has absolute value $1$.  In particular, $f_{\textup{norm}}$ is bounded on $\calD$ (since it is bounded on the fundamental domain $D$), so working with $f_{\textup{norm}}$ yields better numerical stability if $\rho$ is large, since then the contribution from the automorphy factor may otherwise be quite large.  Although $f_{\textup{norm}}$ is no longer holomorphic, this does not affect the above method in any way.  
\end{remark}

\subsection*{Using the Cauchy integral formula}

An alternative way to obtain the coefficients $b_n$ is to use the Cauchy integral formula: for $n \geq 0$, we have
\[ b_n = \frac{1}{2\pi i} \oint \frac{f(z)}{w^{n+1}(1-w)^k}\,dw \]
with the integral a simple contour around $0$.  If we take this contour to be a circle of radius $R\geq \rho$, then again using automorphy, Cauchy's integral is equivalent to one evaluated along a path in $D$ where the approximation $f(z) \approx f_N(z)$ holds; then using techniques of numerical integration we obtain a nontrivial linear relation among the coefficients $b_n$. 

The simplest version of this technique arises by using simple Riemann summation.  Letting $w=\rho e^{i\theta}$ and $dw = i\rho e^{i\theta}\,d\theta$, we have
\[ b_n = \frac{1}{2\pi} \int_0^{2\pi} \frac{f(\rho e^{i\theta})}{(\rho e^{i\theta})^{n}(1-\rho e^{i\theta})^{k}}\,d\theta; \]
again breaking up $[0,2\pi]$ into $Q \in \Z_{\geq 3}$ intervals, letting $w_m=\rho e^{2\pi m i/Q}$ and $w_m'=g_m w_m$ with $w_m' \in D$, and $z_m' = z(p;w_m')$ and $z_m = z(p;w_m)$, we obtain
\[ b_n \approx \frac{1}{Q} \sum_{m=1}^{Q} \frac{f(z_m)}{w_m^n(1-w_m)^k} \approx \frac{1}{Q} \sum_{m=1}^{Q} \frac{ j(g_m, z_m)^{-k} f_N(z_m')}{w_m^n(1-w_m)^k} \]
and expanding $f_N(z)$ we obtain 
\begin{equation} \label{bigdealrelationriemann}
b_n \approx \sum_{r=0}^N K^{\textup{c}}_{nr} b_r
\end{equation}
where (``c'' for Cauchy)
\[ K^{\textup{c}}_{nr} = \frac{1}{Q} \sum_{m=1}^Q \frac{j(g_m,z_m)^{-k}}{w_m^n(1-w_m)^k} (w_m')^r \]
with an error in the approximation that is small if $Q$ is large.  The matrix $K^{\textup{c}}$ with entries $K^{\textup{c}}_{nr}$, with rows indexed by $n=0,\dots,N$ and columns indexed by $r=0,\dots,N$, is obtained from (\ref{bigdealrelationriemann}) by a matrix multiplication:
\begin{equation} \label{matrixmult} Q K^{\textup{c}} = J W'
\end{equation}
where $J$ is the matrix with entries
\begin{equation} \label{defJnm}
 J_{nm}=\frac{j(g_m,z_m)^{-k}}{w_m^n(1-w_m)^k} 
\end{equation}
with $0 \leq n \leq N$ and $1 \leq m \leq Q$ and $W'$ is the Vandermonde matrix with entries
\[ W'_{mr} = (w_m')^r \]
with $1 \leq m \leq Q$ and $0 \leq r \leq N$.  The matrices $J$ and $W'$ are fast to compute, so computing $K^{\textup{c}}$ comes essentially at the cost of one matrix multiplication.  The column vector $b$ with entries $b_n$ satisfies $K^{\textup{c}} b \approx b$, so $b$ is approximately in the kernel of $K^{\textup{c}}-1$.  

\begin{remark}
For large values of $n$, the term $w_m^n$ in the denominator of (\ref{defJnm}) dominates and creates numerical instability.  Therefore, instead of solving for the coefficients $b_n$ we write $b_n'=b_n \rho^n$ and
\[ f(z) = (1-w)^k \sum_{n=0}^{\infty} b_n' (w/\rho)^n \]
and solve for the coefficients $b_n'$.  This replaces $w_m$ by $w_m/\rho$ with absolute value $1$.  We then further scale the relation (\ref{bigdealrelationriemann}) by $1/(K^{\textup{c}}_{nn}-1)$ so that the coefficient of $b_n$ in the relation is equal to $1$; in practice, we observe that the largest entry in absolute value in the matrix of relations occurs along the diagonal, yielding quite good numerical stability.
\end{remark}

Better still is to use Simpson's rule:  Break up the interval $[0,2\pi]$ into $2Q+1\in\Z_{\geq 3}$ intervals of equal length.  With the same notation as above we then have
\begin{equation} \label{simpson}
\begin{aligned}
 b_n \approx \frac{1}{6Q}\sum_{m=1}^{Q}\left(\frac{j(g_{2m-1},z_{2m-1})^{-k}f_N(w'_{2m-1})}{w_{2m-1}^n(1-w_{2m-1})^k}\right. &+
	 4\frac{j(g_{2m},z_{2m})^{-k}f_N(w'_{2m})}{w_{2m}^n(1-w_{2m})^k} \\
&+ \left. \frac{j(g_{2m+1},z_{2m+1})^{-k}f_N(w'_{2m+1})}{w_{2m+1}^n(1-w_{2m+1})^k}\right).
\end{aligned}
\end{equation}
We obtain relations analogous to (\ref{bigdealrelationriemann}) and (\ref{matrixmult}) in a similar way, the latter as a sum over three factored matrices.  Then we have
	\[b_n \approx \sum_{r=0}^NL_{nr}^cb_r\]
where 
\begin{equation*}
\begin{aligned}
	L_{nr}^c = \frac{1}{6Q}\sum_{m=1}^Q\left(\frac{j(g_{2m-1},z_{2m-1})^-k}{w_{2m-1}^n(1-w_{2m-1})^k}(w_{2m-1}')^r\right. &+ 4\frac{j(g_{2m},z_{2m})^-k}{w_{2m}^n(1-w_{2m})^k}(w_{2m}')^r\\
	 &+ \left.\frac{j(g_{2m+1},z_{2m+1})^-k}{w_{2m+1}^n(1-w_{2m+1})^k}(w_{2m+1}')^r\right).
\end{aligned}
\end{equation*}
Then, letting $L^c$ denote the matrix formed by $L_{nr}^c$ with rows indexed by $n$ and columns indexed by $r$, we have that $L^cb\approx b$, where $b$ is the column vector having as its entries the $b_n$, and hence $b$ is in the numerical kernel of $L^c-1$.

Simpson's rule is fast to evaluate and gives quite accurate results and so is quite suitable for our purposes.  In very high precision, one could instead use more advanced techniques for numerical integration; each integral in $D$ can be broken up into a finite sum with contours given by geodesics.

\begin{remark}
The coefficients $b_n$ for $n > N$ are approximately determined by the coefficients $b_n$ for $n \leq N$ by the relation (\ref{bigdealrelationriemann}).  In this way, they can be computed using integration and without any further linear algebra step.
\end{remark}

\subsection*{Hecke operators}

In the special situation where $\Gamma$ is a congruence arithmetic Fuchsian group, we can impose additional linear relations on eigenforms by using the action of the Hecke operators, as follows.

We use the notation introduced in Section 1.  Suppose that $f$ is an eigenform for $\Gamma$ (still of weight $k$) with $\Gamma$ of level $\frakN$.  Let $\frakp$ be a nonzero prime of $\Z_F$ with $\frakp \nmid \frakD\frakN$.  

Using methods of Greenberg and the first author \cite{GV}, we can compute for every prime $\frakp \nmid \frakD\frakN$ the Hecke eigenvalue $a_\frakp$ of an eigenform $f$ using explicit methods in group cohomology.  In particular, for each $\frakp$ we compute elements $\pi_1,\dots,\pi_q \in \PSL_2(K) \subset \PSL_2(\R)$ with $q=N\frakp+1$ so that the action of the Hecke operator $T_\frakp$ is given by
\begin{equation}\label{smalldealrelationhecke}
	(T_\frakp f)(z) = \sum_{i=1}^q j(\pi_i,z)^{-k} f(\pi_i z) = a_\frakp f(z).
\end{equation}
Taking $z=p$, for example, and writing $w_{\frakp,i}=w(\pi_i p)$ for $i=1,\dots,q$ and $w_{\frakp,i}'=g_i w_i \in D$ as before, we obtain
\[ a_\frakp f(0) \approx \sum_{i=1}^q j(\pi_i,p)^{-k} j(g_i, w_{\frakp,i})^{-k} f(z_{\frakp,i}'). \]
Expanding $f$ as a series in $w$ as before, we find
\begin{equation} \label{bigdealrelationhecke}
a_\frakp f(0) \approx \sum_{n=0}^N K_n^{\textup{h}} b_n 
\end{equation}
where (``h'' for Hecke)
\[ K_n^{\textup{h}} = \sum_{i=1}^q j(\pi_i,p)^{-k} j(g_i, w_{\frakp,i})^{-k} (1-w_{\frakp,i}')^k (w_{\frakp,i}')^n. \] 

One could equally well consider the relations induced by plugging in other values of $z$, but in our experience are not especially numerically stable; rather, we use a few Hecke operators to isolate a one-dimensional subspace in combination with those relations coming from modularity.

\begin{remark}
Alternatively, having computed the space $M_k(\Gamma)$, one could turn this idea around and use the above relations to compute the action of the Hecke operators purely analytically!  For each $f$ in a basis for $M_k(\Gamma)$, we evaluate $T_\frakp f$ at enough points to write $T_\frakp f$ in terms of the basis, thereby giving the action of $T_\frakp$ on $M_k(\Gamma)$.
\end{remark}

\subsection*{Derivatives}


Thus far, to encode the desired relationships on the coefficients, we have used the action of the group (i.e. automorphy) and Hecke operators.  We may obtain further relationships obtained from the Shimura-Maass derivatives. 
The following lemma describes the action of $\partial$ on power series expansions (see also Datskovsky and Guerzhoy \cite[Proposition 2]{DatsGuerz}).

\begin{lemma}
	Let $f\in M_k(\Gamma)$ and $p\in \calH$.  Then 
		$$(\partial^m f)(z) = \sum_{r=0}^m\begin{pmatrix}m\\r\end{pmatrix}\frac{(m+r)_{k-r}s_{m-r}(w)}{(4\pi)^{m-r}}(1-w)^{k+2r}\sum_{n=0}^{\infty}\frac{(\partial^{n+r} f)(p)}{n!}(-4\pi y)^n w^n$$
where $y=\impart p$ and
	$$s_n(w) = \sum_{t=0}^n(-1)^{n-t}\begin{pmatrix}n\\t\end{pmatrix}\frac{(1-w)^t}{y^t \impart(z)^{n-t}}.$$
\end{lemma}
\begin{proof}
The result holds for $m=1$ by Lemma \ref{powcn}.  Now suppose that the Lemma holds up to some positive integer $m \geq 1$.  We have
\[ (\partial^{m+1} f)(z) = \partial_{k+2m}(\partial^m f)(z) = \left(\frac{1}{2\pi i}\frac{d}{dz} - \frac{k+2m}{4\pi \impart(z)}\right)(\partial^m f)(z).  \] 
Hence, after substituting for $(\partial^m f)(z)$ and observing that $dw/dz = (1-w)^2/(2iy)$, we obtain an expression for $(\partial^{m+1} f)(z)$, in which we observe that for all $n\in\Z_{\geq0}$
	$$\frac{1}{2\pi i}\frac{ds_n(w)}{dz} = \frac{1}{4\pi}\sum_{t=0}^n(-1)^{n-t}\begin{pmatrix}n\\t\end{pmatrix}\frac{(1-w)^t}{y^t\impart(z)^{n-t}}\left[\frac{t(1-w)}{y} + \frac{n-t}{\impart(z)}\right].$$
Utilizing this in conjunction with
\begin{align*}
	\frac{1}{2\pi i}\frac{ds_{m-r}(w)}{dz} &+ \frac{(k+2r)s_{m-r}(w)(1-w)}{4\pi y} - \frac{(k+2m)s_{m-r}(w)}{4\pi Im(z)}\\
	& = \frac{(k+m+r)s_{m+1-r}(w)}{4\pi},
\end{align*}
we obtain the result for $\partial^{m+1}$, from which the result follows by induction.
\end{proof}

For an eigenform $f \in M_k(\Gamma)$, we can further use the Hecke operators in conjunction with the Shimura-Maass derivatives.

\begin{proposition}[Beyerl, James, Trentacose, Xue \cite{BJCX}]
Let $f\in M_k(\Gamma)$.  Then $\partial_k^{m}(f)$ is a Hecke eigenform if and only if $f$ is and eigenform; if so, and $a_n$ is the eigenvalue of $T_n$ associated to $f$, then the eigenvalue of $T_n$ associated to $\partial_k^{m}(f)$ is $n^m a_n$.
\end{proposition}


\subsection*{Computing the numerical kernel}

Following the previous subsections, we assemble linear relations into a matrix $A$ with $M$ rows and $N+1$ columns such that $Ab \approx 0$.  We now turn to compute the numerical kernel of $A$.  We may assume that $M \geq N$ but do not necessarily require that $M=N$.

Suppose first we are in the case where the space of forms of interest is one-dimensional.  This happens when, for example, $X=\Gamma \backslash \calH$ has genus $g=1$ and $k=2$; one can also arrange for this as in the previous section by adding linear relations coming from Hecke operators.  Then since $0$ always belongs to the numerical kernel, we can dehomogenize by setting $b_0=0$; then letting $A'$ be the $M \times N$ matrix consisting of all columns of $A$ but the first column, let $b'$ be the column vector with unknown entries $b_1,\dots,b_N$, and let $a$ be the first column of $A$.  Then $A'b'=-a$, and we can apply any numerically stable linear algebra method, an LU decomposition for example, to find a (least-squares) solution.  

There is a more general method to compute the entire numerical kernel of $A$ which yields more information about the numerical stability of the computation.  Namely, we compute the singular value decomposition (SVD) of the matrix $A$, writing 
$$A = USV^*$$
where $U$ and $V$ are $M \times M$ and $(N+1)\times(N+1)$ unitary matrices and $S$ is diagonal, where $V^*$ denotes the conjugate transpose of $V$.  The diagonal entries of the matrix $S$ are the \emph{singular values} of $A$, the square roots of the eigenvalues of $A^*A$, and may be taken to occur in decreasing magnitude.  Singular values that are approximately zero correspond to column vectors of $V$ that are in the numerical kernel of $A$, and one expects to have found a good quality numerical kernel if the other singular values are not too small.

\subsection*{Confirming the output}

We have already mentioned several ways to confirm that the output looks correct.  The first is to simply decrease $\eps$ and see if the coefficients $b_n$ converge.  The second is to look at the singular values to see that the approximately nonzero eigenvalues are sufficiently large (or better yet, that the dimension of the numerical kernel is equal to the dimension of the space $M_k(\Gamma)$, when it can be computed using other formulas).

More seriously, we can also check that the modularity relations (\ref{bigdealrelation}) hold for a point $w \in \calD$ with $|w| \leq \rho$ but $w \not \in D$.  (Such points always exist as the fundamental domain $D$ is hyperbolically convex.)  This test is quick and already convincingly shows that the computed expansion transforms like a modular form of weight $k$.

Finally, when $f$ is an eigenform for a congruence group $\Gamma$, we can check that $f$ is indeed numerically an eigenform (with the right eigenvalues) and that the coefficients, when normalized as (\ref{normalizeOmega}), appear to be algebraic using the LLL-algorithm \cite{LLL}.

\section{Results}

In this section, we present four examples to demonstrate our method.

\subsection*{Example 1}

We begin by computing with a classical modular form so that we can use its $q$-expansion to verify that the power series expansion is correct.  However, our method does not work well in this non-cocompact situation, so we prepare ourselves for poor accuracy.

Let $f\in S_2(\Gamma_0(11))$ be the unique normalized eigenform of weight $2$ and level $11$, defined by
\[ f(z) = q\prod_{n=1}^{\infty}(1-q^n)^2(1-q^{11n})^2  = q - 2q^2 - q^3 + 2q^4 + ... = \sum_{n=1}^{\infty}a_nq^n \]
where $q = e^{2\pi i z}$.  We choose a CM point (Heegner point) for $\Gamma_0(11)$ for the field $K=\Q(\sqrt{-7})$ with absolute discriminant $d=7$, namely $p=(-9+\sqrt{-7})/22$.  We find the power series expansion (\ref{powseries}) written (\ref{normalizeOmega}) directly using Lemmas \ref{powcn} and \ref{Dvspartial}: we obtain
\begin{align*} 
f(z) &= (1-w)^2 \sum_{n=0}^{\infty} b_n w^n = 
f(p) (1-w)^2 \sum_{n=0}^{\infty} \frac{c_n}{n!} (\Theta w)^n  \\
  &=-\sqrt{3+4\sqrt{-7}} \Omega^2 (1-w)^2 \left(1 + \Theta\omega + \frac{5}{2!} (\Theta w)^2 - \frac{123}{3!} (\Theta w)^3 \right. \\
& \qquad\qquad\qquad\qquad\qquad\qquad\qquad \left. - \frac{59}{4!} (\Theta w)^4 - \frac{6435}{5!} (\Theta w)^5 + \ldots  \right)
\end{align*}
where 
\begin{equation} \label{whatistheta}
 \Theta = -4\pi y \frac{(\partial f)(p)}{f(p)} = \frac{-4+2\sqrt{-7}}{11} \pi \Omega^2 
\end{equation}
and
\begin{align*} 
\Omega &= \frac{1}{\sqrt{2d\pi}} \left(\prod_{j=1}^{d-1} \Gamma(j/d)^{(-d/j)}\right)^{1/2h(d)} \\
&=\frac{1}{\sqrt{14\pi}} \left(\frac{\Gamma(1/7)\Gamma(2/7)\Gamma(4/7)}{\Gamma(3/7)\Gamma(5/7)\Gamma(6/7)}\right)^{1/2} = 0.5004912\ldots,
\end{align*}
and $-\sqrt{3+4\sqrt{-7}} = -2.6457513\ldots + 2i$.  The coefficients $c_n$ are found always to be integers, and so are obtained by rounding the result obtained by numerical approximation; however, we do not know if this integrality result is a theorem, and consequently this expansion for $f$ is for the moment only an experimental observation.

We now apply our method and compare.  We first compute a fundamental domain for $\Gamma_0(11)$ with center $p$, shown in Figure \ref{funddom-ex1}.

\begin{figure}[t]
\sidecaption[t] \centering
\includegraphics[scale=1]{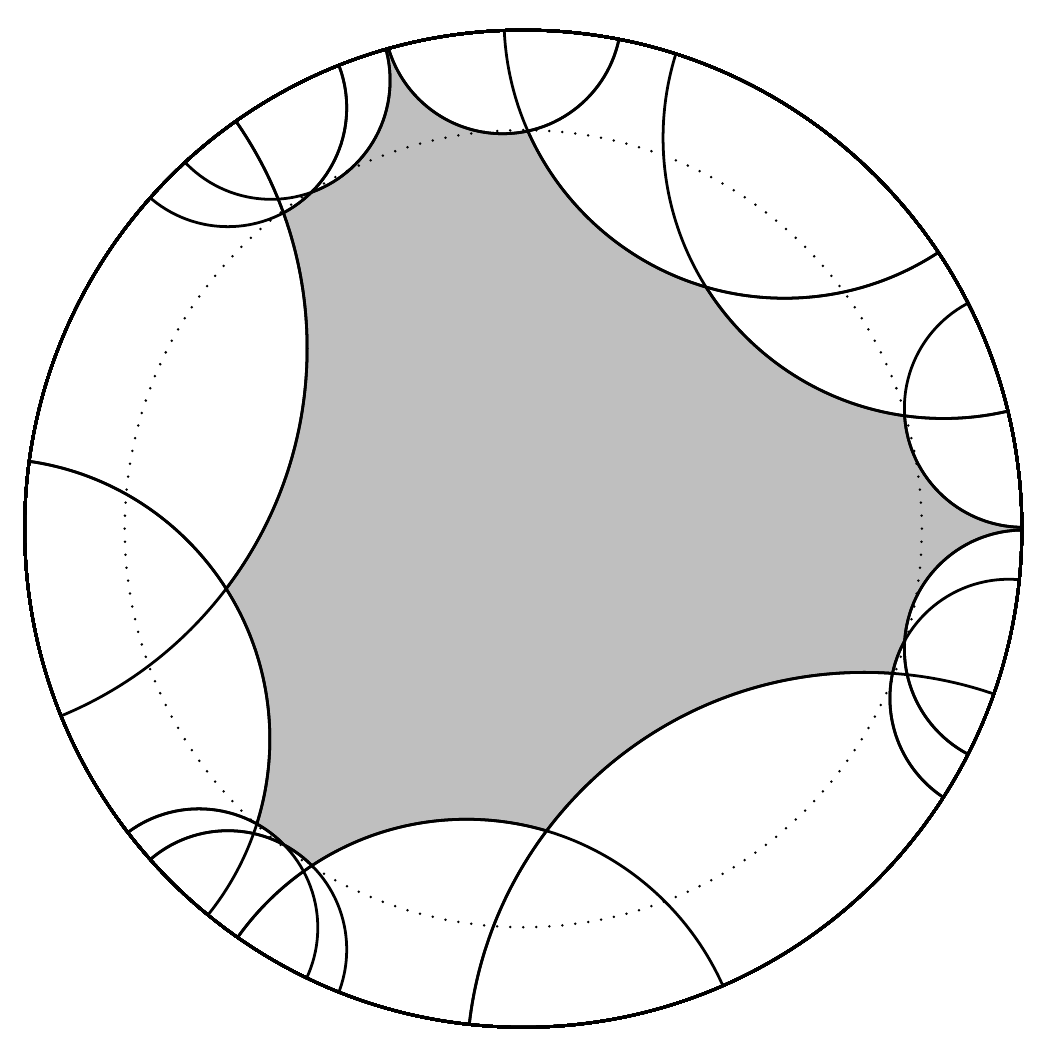}
\caption{A fundamental domain for the modular curve $X_0(11)$}
\label{funddom-ex1}
\end{figure}

Now we take $\epsilon = 10^{-20}$.  In this situation, the fundamental domain is not compact, so we must adapt our method.  We choose a radius $\rho$ for which the image under $\Gamma_0(11)$ to the fundamental domain will yield nontrivial relations; we choose $\rho=0.85$.  Then by (\ref{valueforN}), we may take $N=300$ (adding some for good measure) for the estimate $f(z) \approx f_N(z)$ whenever $|w(z)|\leq \rho$.

We compute the relations (\ref{bigdealrelationriemann}) from the Cauchy integral formula for radius $R$, and we add the relations (\ref{bigdealrelationhecke}) coming from Hecke operators for the primes $2,3,5,7,13$.  (The fact that the image of a point in the fundamental domain under the Hecke operators can escape to a point in the fundamental domain but with radius $> R$ precludes the further use of higher Hecke operators.)

We compute the SVD of this matrix of relations and find the smallest singular value to be $<\eps$; the next largest singular value is only $4$ times as large, so also arguably negligible.  In any case, the associated element of the numerical kernel corresponding to the smallest singular value yields a series $(1-w)^k \sum_{n=0}^{N} \widetilde{b}_n w^n$ with $R^n|\widetilde{b}_n-b_n| < 10^{-9}$ for the first $10$ coefficients but increasingly inaccurate for $n$ moderate to large.  Given that we expected unsatisfactory numerical results, this test at least indicates that our matrix of relations (coming from both the Cauchy integral formula and Hecke operators) passes one sanity check.

\subsection*{Example 2} 

For a comparison, we now consider the well-studied example arising from the $(2,4,6)$-triangle group associated to the quaternion algebra of discriminant $6$, referenced in the introduction.  

We follow Bayer \cite{Bayer}, Bayer and Travesa \cite{BayerTravesa1,BayerTravesa2}, and Baba and Granath \cite{BabaGranath}; for further detail, we refer to these articles.  

Let $B=\quat{3,-1}{\Q}$ be the quaternion algebra of discriminant $6$ over the rationals $\Q$, so that $\alpha^2=3$, $\beta^2=-1$, and $\beta\alpha=-\alpha\beta$.  A maximal order $\calO \subseteq B$ is given by
\[ \calO=\Z \oplus \alpha\Z \oplus \beta\Z \oplus \delta \Z \]
where $\delta=(1+\alpha+\beta+\alpha\beta)/2$.  We have the splitting
\begin{align*}
\iota_\infty:B &\hookrightarrow \M_2(\R) \\
\alpha,\beta &\mapsto \begin{pmatrix} \sqrt{3} & 0 \\ 0 & -\sqrt{3} \end{pmatrix}, \begin{pmatrix} 0 & 1 \\ -1 & 0 \end{pmatrix}. 
\end{align*}
Let $\Gamma=\iota_\infty(\calO_1^\times)/\{\pm 1\}$.  Then $\Gamma$ is a group of signature $(0;2,2,3,3)$ which is an index $4$ normal subgroup in the triangle group 
\[ \Delta(2,4,6)=\la \gamma_2, \gamma_4, \gamma_6 \mid \gamma_2^2=\gamma_4^4=\gamma_6^6 = \gamma_2\gamma_4\gamma_6 = 1 \ra. \]

The point 
\[ p=\frac{\sqrt{6}-\sqrt{2}}{2} i \]
is a CM point of absolute discriminant $d=24$.  The Dirichlet fundamental domain $D(p)$ is then given in Figure \ref{funddom-ex2}.  (A slightly different fundamental domain, emphasizing the relationship to the triangle group $\Delta(2,4,6)$, is considered by the other authors.)  We have $\rho=0.447213\dots$.

\begin{figure}[t]
\sidecaption[t] \centering
\includegraphics[scale=1]{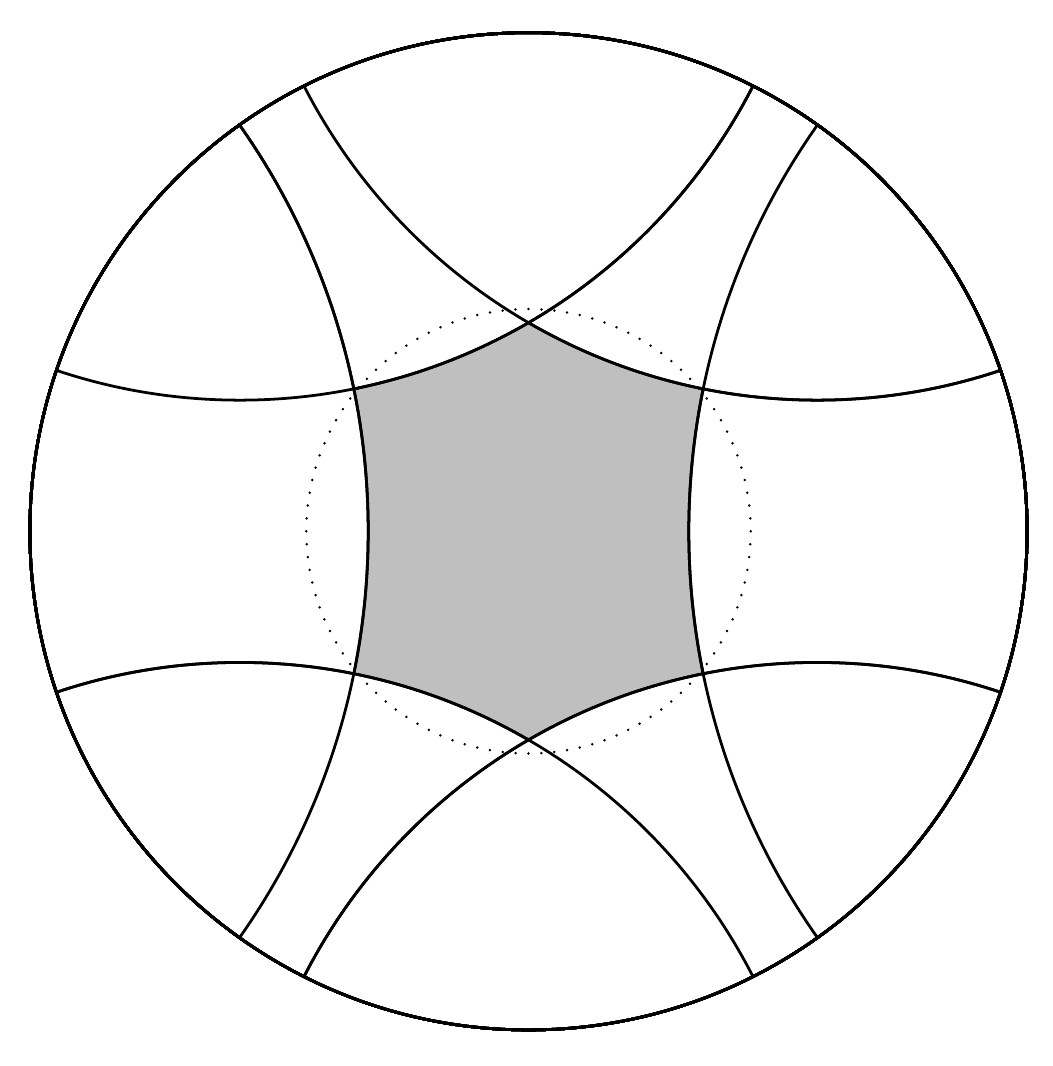}
\caption{A fundamental domain for the Shimura curve $X$ associated to a maximal order in the quaternion algebra of discriminant $6$ over $\Q$}
\label{funddom-ex2}
\end{figure}

We consider the space $S_4(\Gamma)$ of modular forms for $\Gamma$ of weight $4$; by Riemann-Roch, this space has dimension $1$ and so is generated by a modular form $P \in S_4(\Gamma)$ with expansion
\begin{equation}  \label{bnS4}
\begin{aligned}
P(z) = \bigl(\sqrt{-3}(2-\sqrt{3})\Omega^4 \bigr) (1-w)^4 &\left(
\frac{1}{12} + \frac{5}{12} \frac{1}{2!} (\Theta w)^2 - \frac{45}{8} \frac{1}{4!} (\Theta w)^4 + \right. \\
& \quad \left. \frac{555}{4} \frac{1}{6!} (\Theta w)^6 + \frac{57165}{8} \frac{1}{8!} (\Theta w)^8 + \dots \right)
\end{aligned}
\end{equation}
where $\Theta = -4\pi \Omega^2$ and 
\[ \Omega = \frac{1}{\sqrt{2d\pi}} \left(\prod_{j=1}^{d-1} \Gamma(j/d)^{(-d/j)}\right)^{1/2h(d)} = 0.321211772390\ldots \]
with $d=24$.  

\begin{remark}
The algebraic factor $y=\impart(p)=\sqrt{2}(\sqrt{3}-1)/2$ need not be multipled in the period $\Theta$, as expected by (\ref{normalizeOmega}), to obtain coefficients defined over $\Q$.  
\end{remark}

We present the algorithm using the formulas obtained from Simpson's rule.  Let $K^c$ be the $(N+1)\times (N+1)$ matrix obtained from Simpson's rule.  We then take a square $N\times N$ submatrix of full rank and set $k_0$ to be the removed row vector to obtain a system $Ab=k_0$; we wish to solve for $b$.  An LU decomposition and back substitution is then used to obtain coefficients $b_n$ normalized so that $b_0 = 1$.  Let $b_{n,\textup{exact}}$ denote the exact coefficients in (\ref{bnS4}) above, renormalized so that $b_{0,\textup{exact}} = 1$.  We set $Q = 2N$ as this choice provides a good approximation to the integral.  Then we obtain the following errors.

\begin{table}
\centering
\caption{Example 2 Results}
\label{tab:1}       
%
%
\begin{tabular}{p{2cm}p{2.4cm}p{2.9cm}}
\hline\noalign{\smallskip}
$N$ & $\rho|b_{1,\textup{exact}}-b_1|$ & $\max_n \rho^n|b_{n,\textup{exact}}-b_n|$ \\
\noalign{\smallskip}\svhline\noalign{\smallskip}
	35 & $10^{-13}$ & $10^{-13}$\\
	70  & $10^{-23}$ & $10^{-22}$\\
	 140 & $10^{-47}$ & $10^{-47}$\\
\noalign{\smallskip}\hline\noalign{\smallskip}
\end{tabular}
\end{table}

We find that with a fixed radius, we obtain better precision in the answer as $N$ is increased.  That is, the results are completely determined, for a fixed radius and precision, by the degree of the approximating polynomial.

\begin{remark} In contrast to the non-cocompact case, there is no gain in the accuracy of the results from taking a larger radius.  Hence it suffices to fix the radius $\rho$. 
\end{remark}

\subsection*{Example 3}

Next, we work with an arithmetic group over a totally real field, and compute the image of a CM point under the Shimura curve parametrization of an elliptic curve.  Let $F=\Q(a)=\Q(\sqrt{5})$ where $a^2+a-1=0$, and let $\Z_F$ be its ring of integers.  Let $\frakp=(5a+2)$, so $N\frakp=31$.  Let $B$ be the quaternion algebra ramified at $\frakp$ and the real place sending $\sqrt{5}$ to its positive real root: we take $B=\quat{a,5a+2}{F}$.  We consider $F \hookrightarrow \R$ embedded by the second real place, so $a=(1-\sqrt{5})/2=-0.618033\ldots$.  

A maximal order $\calO \subset B$ is given by
\[ \calO = \Z_F \oplus \Z_F \alpha \oplus \Z_F \frac{a+a\alpha-\beta}{2} \oplus \Z_F \frac{(a-1)+a\alpha-\alpha\beta}{2}. \]
Let $\iota_\infty$ be the splitting at the other real place given by
\begin{align*}
\iota_\infty : B &\hookrightarrow \M_2(\R) \\
\alpha,\beta &\mapsto 
\begin{pmatrix}
0 & \sqrt{a} \\ \sqrt{a} & 0
\end{pmatrix}, \ 
\begin{pmatrix}
\sqrt{5a+2} & 0 \\ 0 & -\sqrt{5a+2}
\end{pmatrix}
\end{align*}

Let $\Gamma=\iota_{\infty}(\calO_1^\times)/\{\pm 1\} \subseteq \PSL_2(\R)$.  Then $\Gamma$ has hyperbolic area $1$ (normalized so that an ideal triangle has area $1/2$) and signature $(1;2^2)$, so $X=\Gamma \backslash \calH$ can be given the structure of a compact Riemann surface of genus $1$.  Consequently, the space $S_2(\Gamma)$ of modular forms on $\Gamma$ of weight $2$ is $1$-dimensional, and it is this space that we will compute.

The field $K=F(\sqrt{-7})$ embeds in $\calO$ with 
\[ \mu = -\frac{1}{2} - \frac{5a+10}{2}\alpha - \frac{a+2}{2}\beta + \frac{3a-5}{2}\alpha\beta \in \calO \]
satisfying $\mu^2 + \mu + 2 =0$ and $\Z_F[\mu]=\Z_K$ the maximal order with class number $1$.  We take $p=-3.1653\ldots + 1.41783\ldots\sqrt{-1} \in \calH$ to be the fixed point of $\mu$, a CM point of discriminant $-7$.

We compute a Dirichlet fundamental domain $D(p)$ for $\Gamma$ as in Figure \ref{funddom-ex31}.

\begin{figure}[t]
\sidecaption[t] \centering
\includegraphics[scale=1]{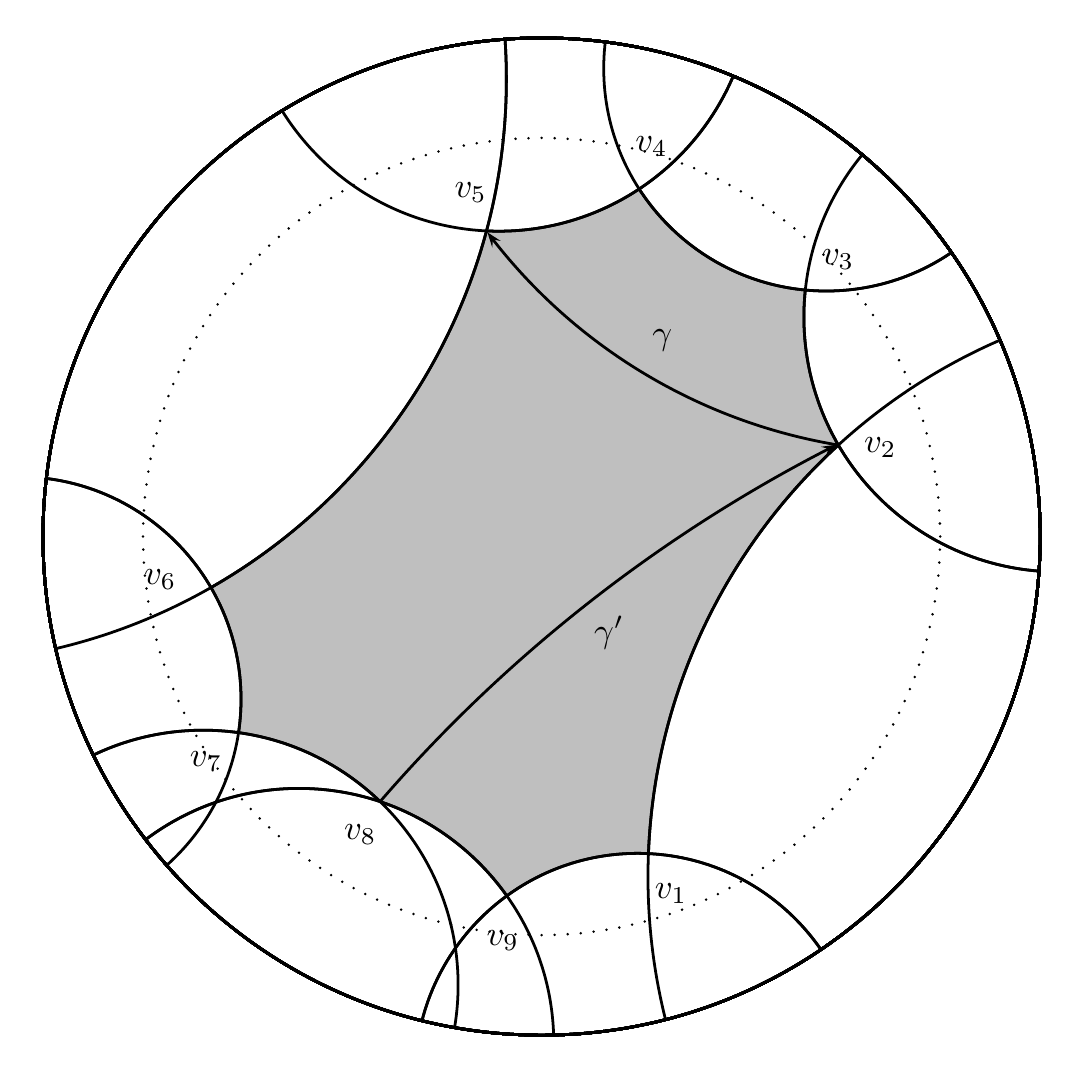}
\caption{A fundamental domain for the Shimura curve $X$ associated to a maximal order in the quaternion algebra ramified at a prime above $31$ and the first real place over $F=\Q(\sqrt{5})$}
\label{funddom-ex31}
\end{figure}

We have $\rho=0.71807\ldots$ so for $\eps=10^{-20}$ we take $N=150$, and we take $R=0.8$.  As $S_2(\Gamma)$ is $1$-dimensional, we use only the relations coming from the Cauchy integral formula (\ref{bigdealrelationriemann}) and reserve the relations from the Hecke operators as a check.  The $(N+1) \times (N+1)$-matrix has largest singular value $4.01413\ldots$ and one singular value which is $<\eps$---the next largest singular value is $0.499377\ldots$, showing that the numerical kernel is one-dimensional.  

Computing this kernel, we find numerically that
\begin{equation} \label{ex3}
\begin{aligned} 
f(z) &= (1-w)^2 \left(1 + (\Theta w) - 
\frac{70a + 114}{2!} (\Theta w)^2 - 
\frac{8064a + 13038}{3!} (\Theta w)^3 + \right. \\
& \quad\quad \frac{174888a + 282972}{4!} (\Theta w)^4 - 
\frac{13266960a + 21466440}{5!} (\Theta w)^5 - \\
&\quad\quad \frac{1826784288a + 2955799224}{6!} (\Theta w)^6 - \\
&\quad\quad \left. \frac{2388004416a + 3863871648}{7!} (\Theta w)^7 + \dots \right) 
\end{aligned}
\end{equation}
where 
\[ \Theta = 0.046218579529208499918\ldots - 0.075987317531832568351\ldots \sqrt{-1} \]
is a period akin to (\ref{whatistheta}) and related to the CM abelian variety associated to the point $p$.  (We do not have a good way to scale the function $f$, so we simply choose $f(p)=1$.)  The apparent integrality of the numerators allows the LLL-algorithm to identify these algebraic numbers even with such low precision.

We then also numerically verify (to precision $\eps$) that this function is an eigenfunction for the Hecke operators for all primes $\frakl$ with $N\frakl < 30$.  This gives convincing evidence that our function is correct.

We further compute the other embedding of this form by repeating the above computation with conjugate data: we take an algebra ramified at $\frakp$ and the other real place.  We find the following fundamental domain in Figure \ref{funddom-ex32}.

\begin{figure}[t]
\sidecaption[t] \centering
\includegraphics[scale=1]{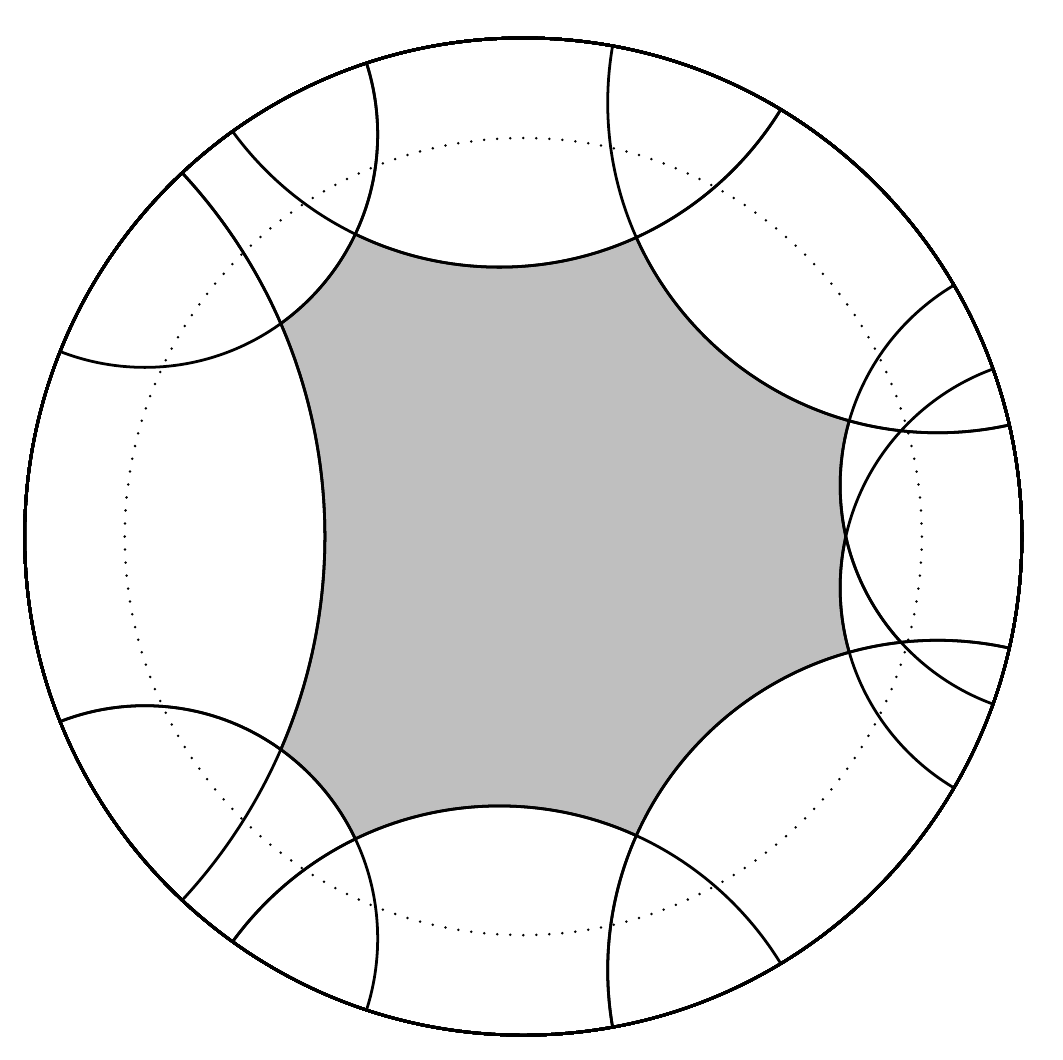}
\caption{A fundamental domain for the Shimura curve $X$ associated to a maximal order in the quaternion algebra ramified at a prime above $31$ and the second real place over $F=\Q(\sqrt{5})$}
\label{funddom-ex32}
\end{figure}

A similar computation yields a matrix of relations with a one-dimensional numerical kernel whose corresponding expansion has coefficients which agree with the conjugates of those in (\ref{ex3}) under  $a \mapsto -(a+1)$, the generator of the group $\Gal(\Q(\sqrt{5})/\Q)$.

Next, we identify the equation of the Jacobian $J$ of the curve $X$ by computing the associated periods.  We first identify the group $\Gamma$ using the sidepairing relations coming from the computation of $D(p)$ \cite{Voight-funddom}: 
\[ \Gamma \cong \la \gamma_1, \gamma_2, \delta_1, \delta_2 \mid \delta_1^2 = \delta_2^2 = 
\gamma_1^{-1} \gamma_2^{-1} \delta_1 \gamma_1 \gamma_2 \delta_2 = 1 \ra \]
where
\begin{align*}
\gamma_1 &= \frac{a+2}{2} - \frac{2a+3}{2}\alpha + \frac{a+1}{2}\alpha\beta \\
\gamma_2 &= \frac{2a + 3}{2} + \frac{7a + 10}{2}\alpha + \frac{a + 2}{2}\beta - (3a + 5)\alpha\beta
\end{align*}
generate the free part of the maximal abelian quotient of $\Gamma$.  The elements $\gamma_1,\gamma_2$ identify vertices: $v_2 \mapsto v_5=\gamma_1(v_2)$ and $v_8 \mapsto v_2=\gamma_2(v_8)$.  Therefore, we compute two independent periods $\omega_1,\omega_2$ (up to scaling) 
\begin{align*}
\omega_1 &= \int_{v_2}^{v_5} f(z) \frac{dw}{(1-w)^2} \approx \left(\sum_{n=0}^N \frac{b_n}{n+1} w^{n+1} \right)\Biggl|_{v_2}^{v_5} = -0.654017\ldots + 0.397799\ldots i \\
\omega_2 &= \int_{v_8}^{v_2} f(z) \frac{dw}{(1-w)^2} = 0.952307\ldots + 0.829145\ldots i.
\end{align*}
We then compute the $j$-invariant 
\[ j(\omega_1/\omega_2)=-18733.423\ldots = -\frac{11889611722383394a + 8629385062119691}{31^8} \]
and by computing the values of the Eisenstein series $E_4(\omega_1/\omega_2)$ and $E_6(\omega_1/\omega_2)$, or by twisting a curve with $j$-invariant $j(\omega_1/\omega_2)$, we identify the elliptic curve $J$ as
\[ y^2 + xy - ay = x^3 - (a - 1)x^2 - (31a + 75)x - (141a + 303). \]
We did this by recognizing this algebraic number using LLL and then confirming by computing the conjugate of $j$ under $\Gal(\Q(\sqrt{5})/\Q)$ as above and recognizing the trace and norm as rational numbers using simple continued fractions.  Happily, we see that this curve also appears in the tables of elliptic curves over $\Q(\sqrt{5})$ computed \cite{Dembele, Stein}.

Finally, we compute the image on $J$ of a degree zero divisor on $X$.  The fixed points $w_1,w_2$ of the two elliptic generators $\delta_1$ and $\delta_2$ are CM points of discriminant $-4$.  Let $K=F(i)$ and consider the image of $[w_1]-[w_2]$ on $J$ given by the Abel-Jacobi map as
\[ \int_{w_1}^{w_2} f(z) \frac{dw}{(1-w)^2} \equiv -0.177051\ldots - 0.291088\ldots i \pmod{\Lambda} \]
where $\Lambda=\Z\omega_1 + \Z\omega_2$ is the period lattice of $J$.  Evaluating the elliptic exponential, we find the point
\[ (-10.503797\ldots, 5.560915\ldots - 44.133005\ldots i) \in J(\C) \]
which matches to the precision computed $\eps=10^{-20}$ the point
\[ Y = \left(\frac{-81a-118}{16}, \frac{(358a+1191)i + (194a+236)}{64} \right) \in J(K) \cong \Z/4\Z \oplus \Z \]
which generates $J(K)/J(K)_{\textup{tors}}$.  On the other hand, the Mordell-Weil group of $J$ over $F$ is in fact torsion, with $J(F) \cong \Z/2\Z$, generated by the point $\bigl((28a + 27)/4,-(24a + 27)/8\bigr)$.

\subsection*{Example 4}

To conclude, we compute the equation of a genus $2$ Shimura curve over a totally real field.  (There is nothing special about this curve; we chose it because only it seemed to provide a nice example.)

Let $F=\Q(a)=\Q(\sqrt{13})$ where $w^2-w-3=0$, and let $\Z_F=\Z[w]$ be its ring of integers.  Let $\frakp=(2w-1)=(\sqrt{13})$.  We consider the quaternion algebra $B=\quat{w-2,2w-1}{F}$ which is ramified at $\frakp$ and the real place sending $a \mapsto (1-\sqrt{13})/2=-1.302775\ldots$.  We consider $F \hookrightarrow \R$ embedded by the first real place, so $w=(1+\sqrt{13})/2=2.302775\ldots$.  

A maximal order $\calO \subset B$ is given by
\[ \calO = \Z_F \oplus \Z_F \alpha \oplus \Z_F \frac{w\alpha+\beta}{2} \oplus \Z_F \frac{(w+1)+\alpha\beta}{2}. \]
Let $\iota_\infty$ be the splitting 
\begin{align*}
\iota_\infty : B &\hookrightarrow \M_2(\R) \\
\alpha,\beta &\mapsto 
\begin{pmatrix}
0 & \sqrt{w} \\ \sqrt{w} & 0
\end{pmatrix}, \ 
\begin{pmatrix}
\sqrt{2w-1} & 0 \\ 0 & -\sqrt{2w-1}
\end{pmatrix}
\end{align*}

Let $\Gamma=\iota_{\infty}(\calO_1^\times)/\{\pm 1\} \subseteq \PSL_2(\R)$.  Then $\Gamma$ has hyperbolic area $2$ and signature $(2;-)$, so $X=\Gamma \backslash \calH$ can be given the structure of a compact Riemann surface of genus $2$ and $S_2(\Gamma)$ is $2$-dimensional.  We compute \cite{GV} that this space is irreducible as a Hecke module, represented by a constituent eigenform $f$ with the following eigenvalues:
\[
\begin{array}{c||cccccccc}
\frakp  & (w) & (w-1) & (2) & (2w-1) & (w+4) & (w-5) & (3w+1) & (3w-4) \\
N\frakp & 3 & 3 & 4 & 13  & 17 & 17 & 23 & 23   \\
\hline
\rule{0pt}{2.5ex} 
a_\frakp(f) & -e+1 & e & -2 & -1 & 2e-4 & -2e-2 & -2e+4 & 2e+2 \\
\end{array}  
\] 
Here, the element $e \in \Qbar$ satisfies $e^2-e-5=0$ and the Hecke eigenvalue field $E=\Q(e)$ has discriminant $21$.  Let $\tau \in \Gal(E/\Q)$ represent the nontrivial element; then $S_2(\Gamma)$ is spanned by $f$ and $\tau(f)$, where $a_\frakp(\tau(f))=\tau(a_{\frakp}(f))$.  

We compute a Dirichlet fundamental domain $D(p)$ for $\Gamma$ as in Figure \ref{funddom-ex4}.

\begin{figure}[t]
\sidecaption[t] \centering
\includegraphics[scale=0.9]{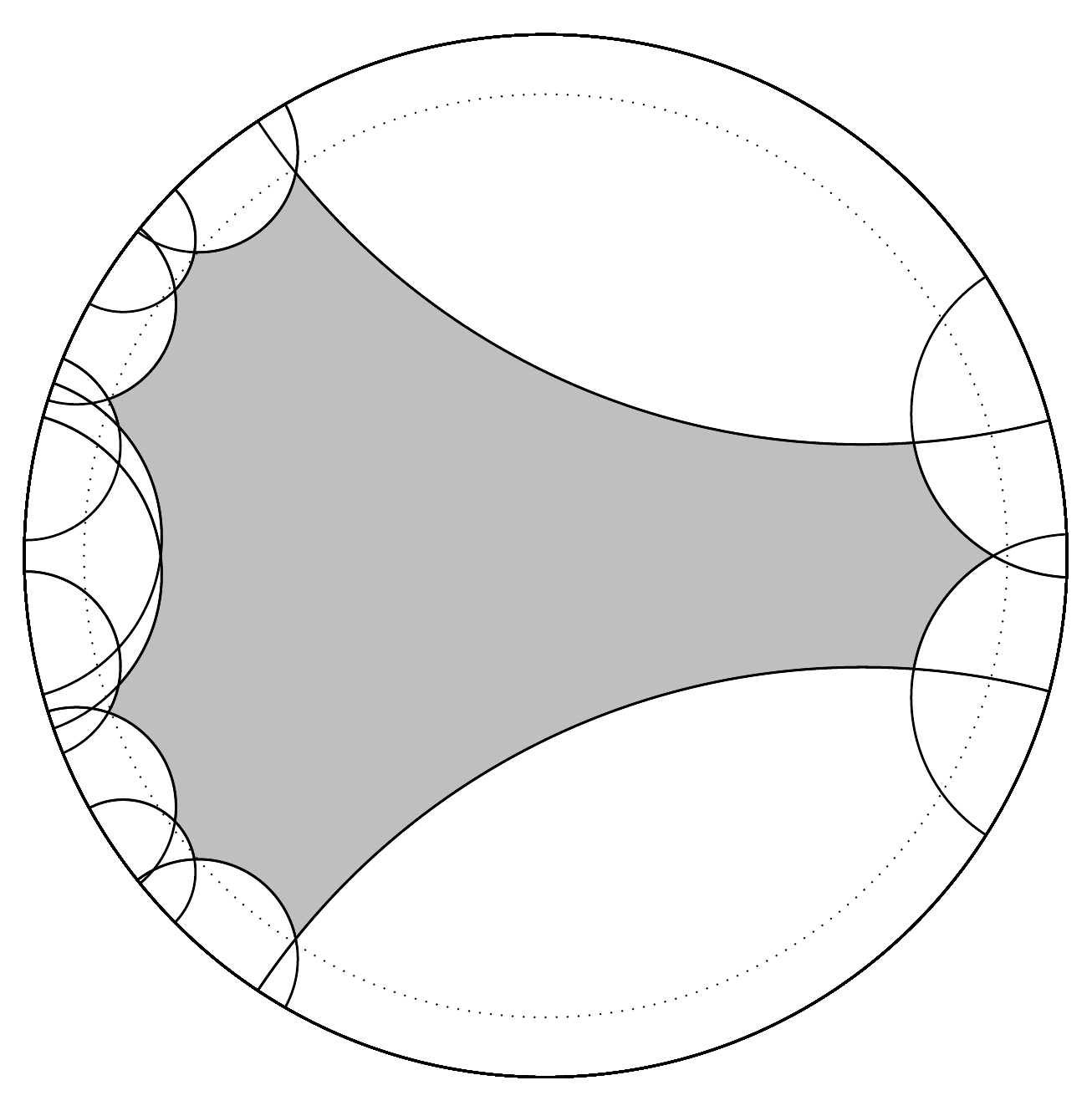}
\caption{A fundamental domain for the Shimura curve $X$ associated to a maximal order in the quaternion algebra ramified at a prime above $13$ and the second real place over $F=\Q(\sqrt{13})$}
\label{funddom-ex4}
\end{figure}

We have $\rho=0.885611\ldots$ so for $\eps=10^{-20}$ we take $N=600$.  For the form $f$ and its conjugate $\tau(f)$, we add the implied relations for the Hecke operators to those from the Cauchy integral.  In each case, we find one singular value $<\eps$ and the next largest singular value to be $>0.2$, so we have indeed found two one-dimensional kernels.

We first consider $g=f+\tau(f)$, and find numerically that
\begin{equation} \label{ex4g}
\begin{aligned} 
g(z) &= (1-w)^2 \left(1 + 117 (\Theta w) 
+ \frac{28496}{2!} (\Theta w)^2
-\frac{14796288}{3!} (\Theta w)^3 \right. \\
& \quad\quad -\frac{3287025664}{4!} (\Theta w)^4 -
\frac{1142501867520}{5!} (\Theta w)^5 \\
&\quad\quad  \left.  -\frac{116349452943360}{6!} (\Theta w)^6 -\frac{420556103693107200}{7!} (\Theta w)^7 + \dots \right) 
\end{aligned}
\end{equation}
where 
\[ \Theta = 0.008587333922373292375569160851\ldots. \]
We see in particular that the coefficients are rational numbers.  

We then compute $h=(f-\tau(f))/\sqrt{21}$ and find
\begin{equation} \label{ex4h}
\begin{aligned} 
h(z) &= \frac{7\sqrt{13}}{3}(1-w)^2 \left(0 - (\Theta w) 
+ \frac{624}{2!} (\Theta w)^2
-\frac{299520}{3!} (\Theta w)^3 \right. \\
& \quad\quad -\frac{9345024}{4!} (\Theta w)^4 +
\frac{180255129600}{5!} (\Theta w)^5 \\
&\quad\quad  \left.  -\frac{160373570273280}{6!} (\Theta w)^6 +\frac{22574212778557440}{7!} (\Theta w)^7 + \dots \right) 
\end{aligned}
\end{equation}

Then, from this basis of differentials, we can use the standard Riemann-Roch basis argument to compute an equation for the curve $X$ (see e.g.\ Galbraith \cite[\S 4]{Galbraith}).  Let $x=g/h$ and $y=x'/h$, where ${}'$ denotes the derivative with respect to $\Theta w$; then since $h$ has a zero of order $1$ at $p$, the function $x$ has a pole of order $1$, and similarly $y$ has a pole of order $3$, therefore by Riemann-Roch we obtain an equation $y^2=q(x)$ with $q(x) \in \Q[x]$: 
\begin{align*} y^2 &= x^6+1950x^5+828919x^4-122128188x^3+3024544159x^2 \\
&\qquad -29677133122x+107045514121.
\end{align*}
This curve has a reduced model
\[ y^2 + y = 7x^5 - 84x^4 + 119x^3 + 749x^2 +    938x + 390 \]
and this reduced model has discriminant $2995508600908518877=7^{10} 13^9$.  

The curve $X$ over $F$, however, has good reduction away from $\frakD=(\sqrt{13})$ \cite{Carayol}.  Indeed, since we don't know the precise period to normalize the forms, we may inadvertently introduce a factor coming from the field $K=F(\sqrt{-7})$ of CM of the point $p$ at which we have expanded these functions as a power series.  In this case, our model is off by a quadratic twist coming from the CM field $K=F(\sqrt{-7})$: twisting we obtain
\begin{equation} \label{Xeq}
X: y^2 + y = x^5 + 12x^4 + 17x^3 - 107x^2 + 134x - 56.
\end{equation}
The only automorphism of $X$ is the hyperelliptic involution, so $\Aut(X)=\Z/2\Z$.

Since $F$ is Galois over $\Q$, has narrow class number $1$, and the discriminant $\frakD$ of $B$ is invariant under $\Gal(F/\Q)$, it follows from the analysis of Doi and Naganuma \cite{DoiNaganuma} that the field of moduli of the curve $X$ is $\Q$.  This does not, however, imply that $X$ admits a model over $\Q$: there is an obstruction in general to a curve of genus $2$ to having a model over its field of moduli (Mestre's obstruction \cite{Mestre}, see also Cardona and Quer \cite{CaQuer}).  This obstruction apparently vanishes for $X$; it would be interesting to understand this phenomenon more generally.

Finally, this method really produces numerically the canonical model of $X$ (equation (\ref{Xeq}) considered over the reflex field $F$), in the sense of Shimura \cite{Shimuracanon}.  In this way, we have answered ``intriguing question'' of Elkies \cite[5.5]{Elkies}.

\begin{acknowledgement}
The authors would like to thank Srinath Baba, Valentin Blomer, Noam Elkies, David Gruenewald, Kartik Prasanna, Victor Rotger, and Frederik Str\"omberg for helpful comments on this work.  The authors were supported by NSF grant DMS-0901971.  
\end{acknowledgement}

\end{document}